\documentclass[11pt,a4paper]{amsart}
\usepackage{amsthm,amsmath,amssymb}
\usepackage[english]{babel}
\usepackage{graphicx}
\usepackage{hyperref}\hypersetup{colorlinks, citecolor=black,filecolor=black,linkcolor=black,urlcolor=black}
\usepackage[shortlabels]{enumitem}
\usepackage{caption}
\usepackage{subcaption}

\newtheorem{thm}{Theorem}
\newtheorem*{thm*}{Theorem}

\newtheorem{defn}{Definition}
\newtheorem{example}{Example}

\title[Projective Equivalence on surfaces of neg. Euler characteristic]{Projectively equivalent Finsler metrics on surfaces of negative Euler characteristic}
\author{Julius Lang}
\address{Affilation and adresses of the author: Friedrich-Schiller University Jena, FMI\\
 Ernst-Abbe-Platz 2, 07743 Jena, Germany\\
 julius.lang@uni-jena.de}

\begin{document}

	\begin{abstract}
	We proof that on a surface of negative Euler characteristic, two real-analytic Finsler metrics have the same unparametrized oriented geodesics, if and only if they differ by a scaling constant and addition of a closed 1-form.
	\end{abstract}
	
	\maketitle

\section{Introduction}
	Two Finsler metrics on the same manifold are called \textit{projectively equivalent}, if they have the same unparametrized, oriented geodesics. It is trivial that any two Finsler metrics $\hat F, \check F: TM \to \mathbb R$ related by $\hat F = \lambda \check F + \beta$, where $\lambda > 0$ is a constant and $\beta$ is a closed 1-form on $M$, are projectively equivalent. In this paper we proof that, under the assumption of real-analicity, on a closed surface of negative Euler characteristic any projective equivalent metrics must be related in this way:
	
	\begin{thm}\label{maintheorem}
	On a surface $\mathcal S$ of negative Euler characteristic, two real-analytic Finsler metrics $\hat F, \check F$ are projectively equivalent, if and only if $\hat F = \lambda \check F + \beta$ for some $\lambda >0$ and a closed 1-form $\beta$.
	\end{thm}

	The corresponding result for Riemannian metrics has been obtained in \cite[Corollary 3]{TopalovMatveevMetricWithErgodic} (see also \cite{TopalovMatveevGeodesicEquivalenceViaIntegrability,MettlerGeodesicRigidity, MettlerPaternainConvexProjective}), where the assumption of real-analicity is not necessary.
	
	The outline of the proof is the following.
	Recall that the geodesic spray $S$ of a Finsler metric $F$ can be obtained as the Hamiltonian vector field of the Hamiltonian system on $T \mathcal S \backslash 0$ with symplectic form $\omega = d \theta$, where $\theta=(\tfrac12 F^2)_{\xi^i} dx^i$ is the \textit{Hilbert 1-form}, and Hamiltonian $\tfrac12 F^2$. A Hamiltonian system on $T \mathcal S \backslash 0$ is called \textit{integrable}, if there exists a function $I: T \mathcal S \backslash 0 \to \mathbb R$ constant along the integral curves of the Hamiltonian system and such that the differentials of the Hamiltonian and $I$ are linearly independent on an open and dense subset. A function $I$ constant along the integral curves of $S$ is called an \textit{integral}.
	
	As we are in dimension 2, in all local coordinates the fiber-Hessians of two metrics $\hat F$ and $\check F$ given by $(\hat h_{ij})=(\hat F_{\xi^i \xi^j})$ and $( \check h_{ij})=(\check F_{\xi^i \xi^j})$ must be proportional at every vector $(x,\xi) \in T\mathcal S \backslash 0$ and the factor of proportionality $I: T\mathcal S\backslash 0 \to \mathbb R$ is given by 
		$I(x,\xi)= \frac{\operatorname{tr} \hat h}{\operatorname{tr}  \check h}|_{(x,\xi)}$.
	Indeed by 1-homogeneity of a metric $F$ we have $h_{ij}|_{(x,\xi)}\xi^j = 0$ and it follows that 
		\begin{equation}\label{EquationHessian}
		h|_{(x,\xi)} = \frac{\operatorname{tr} h|_{(x,\xi)}}{(\xi^1)^2 + (\xi^2)^2} \begin{pmatrix}(\xi^2)^2 & - \xi^1 \xi^2 \\ -\xi^1 \xi^2 & (\xi^1)^2	\end{pmatrix}.
		\end{equation}	
	When the metrics $\hat F, \check F$ are projectively equivalent, we show: 
	\begin{enumerate}[(a)]
		\item \label{FactIisIntegral} The factor of proportionality $I$ is independent of the choice of local coordinates and an integral for the geodesic flow of both metrics.
	\end{enumerate}		
	 This follows from an investigation of the so called Rapscak conditions for projective metrization similar to \cite{CrampinMestdagSaundersTheMultiplierApproach}. The proof is given in section \ref{SectionIisIntegral}. 
	 Alternatively, one can obtain the integral by a general construction for \textit{trajectory equivalent} Hamiltonian systems, similarly to \cite{MatveevTopalovTrajectoryEquivalence}.
	
	In order to show that the integral $I$ must be constant on $T \mathcal S \backslash 0$, we combine two classical results from the theory of integrable systems:
	\begin{enumerate}[(a)]\setcounter{enumi}{1}
		\item \label{FactPositiveEntropy} The \textit{topological entropy} of the geodesic flow of a metric on a compact manifold, whose fundamental group is \textit{of exponential growth}, is positive.
		\item \label{FactEntropyIsZero} If a 4-dimensional Hamiltonian system is integrable by a real-analytic integral independent of the Hamiltonian, then its \textit{topological entropy} vanishes.
	\end{enumerate}
	
	We recall the definition of \textit{topological entropy} of a Hamiltonian system and \textit{exponential growth} of a group in section \ref{SectionPositiveEntropy}. Proposition \ref{FactPositiveEntropy} was proven for the geodesic flow of Riemannian metrics and it seems to be commonly accepted that it generalizes straight-forwardly to the Finsler case. Nonetheless, in section \ref{SectionPositiveEntropy} we give a proof based on the classical proofs for the Riemannian case from \cite{DinaburgOnTheRelationsAmongVariousEntropy,ManningTopologicalEntropyForGeodesicFlows}.
	A proof of \ref{FactEntropyIsZero} can be found in \cite{PaternainEntropyAndCompletelyIntegrableHamiltonianSystems}. A similar argument to show that an integral must be constant was used in \cite{PaternainFinslerStructures}.

	The propositions \ref{FactIisIntegral}, \ref{FactPositiveEntropy} and \ref{FactEntropyIsZero} imply Theorem \ref{maintheorem}:	
	
	\begin{proof}[Proof of Theorem \ref{maintheorem}]
	Let $\hat F,\check F$ be projectively equivalent real-analytic Finsler metrics on a surface $\mathcal S$ of negative Euler characteristic. Then the fundamental group of $\mathcal S$ is of exponential growth. By \ref{FactPositiveEntropy}, the geodesic flow of $\hat F$ has positive entropy and thus by \ref{FactEntropyIsZero} the differential of any real-analytic integral must be proportional to the differential of $\tfrac12 \hat F^2$ at least on a set $A \subseteq T \mathcal S \backslash 0$ admitting an accumulation point. By \ref{FactIisIntegral} the function $I$ is such an integral and by homogeneity $V(\tfrac12 \hat F^2)|_{(x,\xi)} = \hat F^2(x,\xi) \not=0$ and $V(I)|_{(x,\xi)}=0$, where $V= \xi^i \partial_{\xi^i}$ is the vertical vector field. Hence the differential of $I$ must vanish on $A$ and by analicity everywhere. Thus $I$ must be a constant $\lambda$ on $T \mathcal S \backslash 0$.
	
	This implies that $\operatorname{tr}(\hat h) = \lambda \operatorname{tr} (\check h)$. As the Hessians have only one independent component - see equation (\ref{EquationHessian}), it follows that $\hat h_{ij} = \lambda \check h_{ij}$ and thus $\hat F = \lambda \check F + \beta$ for some 1-form $\beta$ on $\mathcal S$. But as $\hat F, \check F$ are projectively equivalent, so are $\lambda \hat F,  \check F$, which differ by the 1-form $\beta$, which then must be closed (see Example \ref{ExampleTrivialEquivalence}). 
	\end{proof}

	The assumption of real-analicity is necessary: On any closed surface there are (non real-analytic) projectively equivalent metrics that are not related by scaling and addition of a closed 1-form:
	\begin{example}Let $\hat F$ be the round metric on $S^2$. We claim that there is a smooth metric $\check F$ on $S^2$ projectively equivalent to $\hat F$, such that
	\begin{itemize}
		\item $\hat F$ and $\check F$ coincide over an open, non-empty set $V \subseteq S^2$,
		\item but are not related by $\hat F = \lambda F + \beta$ for any $\lambda > 0$ and any 1-form $\beta$.
	\end{itemize}
	Then by attaching orientable or non-orientable handles to $S^2$ (Figure \ref{subfig1}) in the set $V$, we obtain projectively equivalent metrics on any closed surface, that are not related by scaling and addition of a closed 1-form.
	
	The metric $\check F$ can be constructed by the method from \cite{alvarez,Arcostanzo}: Suppose the space of unparametrized geodesics of a reversible metric on a surface forms a smooth manifold endowed with a positive measure. Define the distance $d: \mathcal S \times \mathcal S \to \mathbb R$ of two points on the surface as the measure of curves intersecting the unique shortest geodesic segment conencting the points (Figure \ref{subfig2}). Then the unparametrized geodesics of the original metric are shortest for the constructed distance function. Then the function  $F(x,\xi):= \frac{d}{dt}|_{t=0}d(x,c(t))$, where $c$ is any curve such that $c(0)=x$ and $\dot c(0)=\xi$, is a Finsler metric projectively equivalent to the original metric.
	
	Applying this procedure to the round sphere with a density function \linebreak $\lambda: S^2 \to \mathbb R_{>0}$ satisfying $\lambda(-x)=\lambda(x)$ and identifying an oriented great circle by its normal (Figure \ref{subfig3}), one obtains the following family of Finsler metrics
	$$F(x,\xi) := \tfrac14 \int_{\substack{\eta \in T_x S^2 \\ |\eta|=1}} \lambda  \big(x \times \eta \big)
	 |\nu_\xi| \qquad \text{with }\nu_\xi(\cdot) = \langle \xi, \cdot \rangle,$$
	all projectively equivalent to the round metric, where $x \times \eta$ denotes the cross-product in $\mathbb R^3$ and $\langle \cdot, \cdot \rangle$ the Euclidean inner product.
	 
	 For $\hat \lambda \equiv 1$, we obtain the round metric $\hat F$.
	Let $U = \{(x_1,x_2,x_3) \in S^2 \mid |x_3| \leq \tfrac1{\sqrt 2}\}$, denote by $V$ its complement and choose any $\check \lambda: S^2 \to \mathbb R_{> 0}$  such that $\check\lambda|_U \equiv 1$, but $\check\lambda|_V >1$.
	The obtained metric $\check F$ coincides with $\hat F$ over $V$, because for $(x,\xi) \in TV$ with $|\xi|=1$, the cross product $x \times \xi$ is in $U$ (Figure \ref{subfig3}).
	\begin{figure}[t]
	\begin{subfigure}[c]{0.32\textwidth}
\includegraphics[scale=0.66]{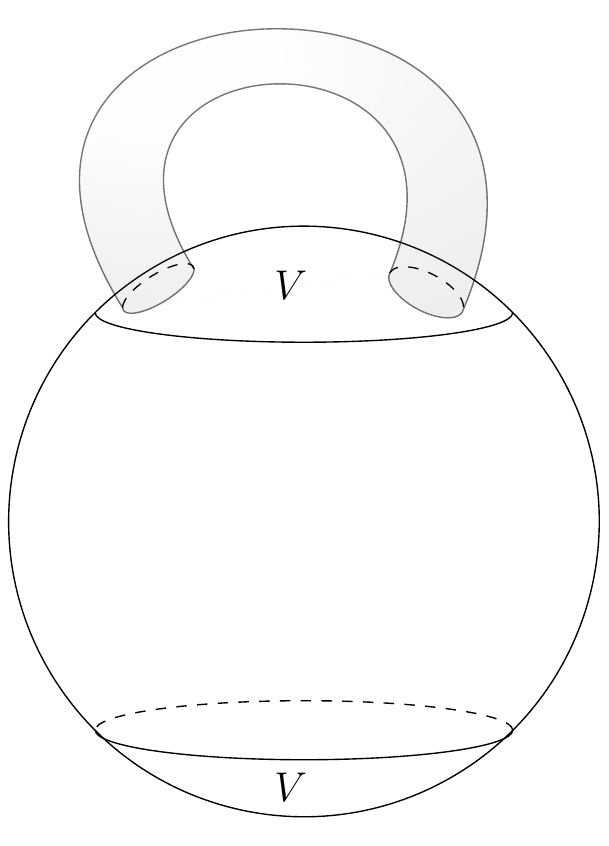}
	\subcaption{Attaching handles in the region over which the metrics coincide.\\}\label{subfig1}
	\end{subfigure}~~
	\begin{subfigure}[c]{0.32\textwidth}
\includegraphics[scale=0.66]{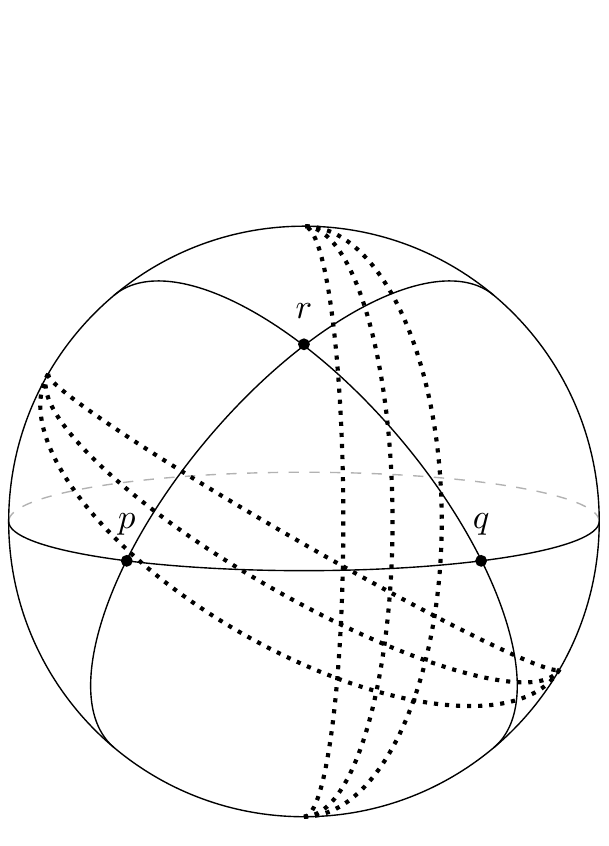}
	\subcaption{The distance $d(p,q)$ is the measure of all curves (dotted) intersecting the great circle segment.}\label{subfig2}
	\end{subfigure}~~
	\begin{subfigure}[c]{0.32\textwidth}
\includegraphics[scale=0.66]{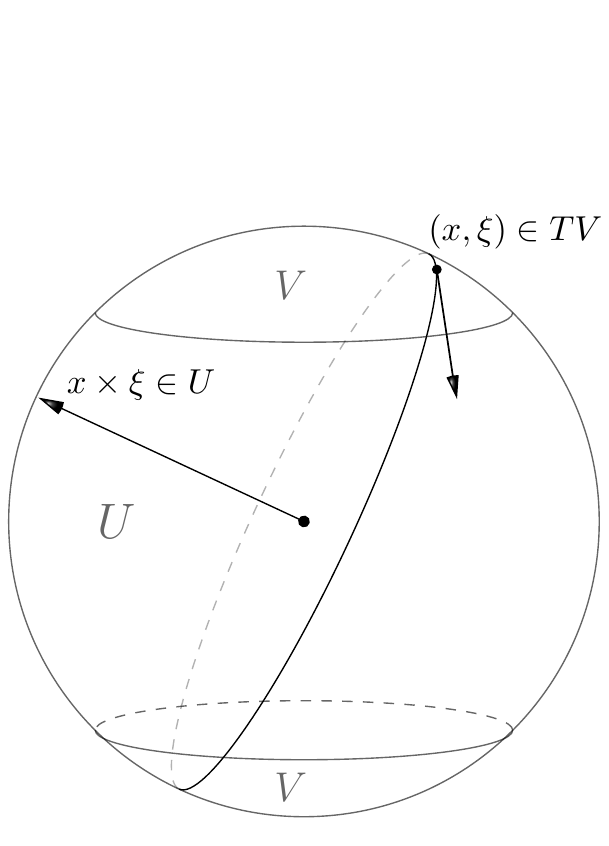}
	\subcaption{The metrics $\check F$ and $\hat F$ coincide on $TV$, because $\check \lambda|_U = \hat \lambda|_U \equiv 1$.\\}\label{subfig3}
	\end{subfigure}
	\caption{Construction of non-trivially projectively equivalent metrics on any closed surface.}
	\end{figure}
	\end{example}

	\section{Projectively equivalent Finsler metrics and proof of \ref{FactIisIntegral}}
	\label{SectionIisIntegral}
	
	Let $M$ be a smooth manifold, $TM \backslash 0$ the tangent bundle with the origins removed and $(x,\xi)$ local coordinates on $TM$.
	\begin{defn}~
	\begin{enumerate}
	\item A \textit{Finsler metric} is a function $TM \to \mathbb R_{\geq 0}$, such that
	\begin{itemize}
	\item $F(x,\lambda \xi) = \lambda F(x,\xi)$ for all $\lambda > 0$.
	\item $F|_{TM \backslash 0}$ is smooth and the matrix $g_{ij}|_{(x,\xi)} := \Big( \tfrac12 \frac{\partial^2 F^2}{\partial{\xi^i} \partial{\xi^j}} \Big)_{ij}$ is positive definite for all $(x,\xi) \in TM \backslash 0$.
	\end{itemize}
	
	\item The geodesics of $F$ are defined as the solutions to the Euler-Lagrange equation $E_i(L,c):=L_{x^i} - \frac{d}{dt}(L_{\xi^i})=0$ for the Lagrangian $L=\tfrac12 F^2$. \\
	\item The geodesic spray of $F$ is the globally defined vector field $S$ on $TM \backslash 0$, whose integral curves projected to $M$ are exactly the geodesics.
	\item Two Finsler metrics $\hat F, \check F$ are projectively equivalent, if any geodesic of $\hat F$ is a geodesic of $\check F$ after an orientation preserving reparametrization.
	\end{enumerate}
	\end{defn}
	The Euler-Lagrange equation for a Finsler metric $F$ and $\tfrac12 F^2$ are related by
	$$E_i(\tfrac 12 F^2,c)
	=F F_{x^i} - \frac{d}{dt}(F F_{\xi^i})	
	= F \cdot E_i(F,c) - \frac{d F}{dt} F_{\xi^i},$$
	so that the solutions to $E_i(F,c)=0$ are all orientation preserving reparametrizations of geodesics of $F$.
	
	Two metrics $\hat F, \check F$ are projectively equivalent, if and only if every geodesic of $\hat F$ is a solution of the Euler-Lagrange equation $E_i(\check F, c)=0$.	Let $S= \xi^i \partial_{x^i} - 2G^i \partial_{\xi^i}$ be the geodesic spray of $\hat F$. Then $\hat F, \check F$ are projectively equivalent, if and only if
	\begin{equation}\label{EquationRapcsak}
	\check F_{x^i} - \check F_{\xi^i x^j} \xi^j +2G^j \check F_{\xi^i \xi^j}=0.
	\end{equation}
	
	\begin{example}\label{ExampleTrivialEquivalence}
	Suppose two Finsler metrics are related by $\check F= \lambda \hat F + \beta$, where $\lambda >0$ and $\beta$ is a 1-form on $M$. Then $\hat F$ and $\check F$ are projectively equivalent, if and only if $\beta$ is closed.
	
	Indeed, let $\beta=\beta_k dx^k$. Using that $\hat F$ is projectively equivalent to itself equation (\ref{EquationRapcsak}) is equivalent to
	$$\beta_{x^i}-\beta_{\xi^i x^j}\xi^j+2G^j \beta_{\xi^i\xi^j}=0,$$
	which is satisfied if and only if $(\beta_j)_{x^i}-(\beta_i)_{x^j} \equiv 0$ for all $i,j$, that is if $\beta$ is closed.
	\end{example}
	
	Let us now proof proposition \ref{FactIisIntegral}.
	\begin{proof}[Proof of \ref{FactIisIntegral}]
		Let $\hat F$ and $\check F$ be projectively equivalent Finsler metrics on a surfaces $\mathcal S$. We shall show that $I: T\mathcal S \backslash 0 \to \mathbb R$ defined in local coordinates by 
		$I(x,\xi):=\frac{\operatorname{tr} \hat h}{\operatorname{tr}\check h}\big|_{(x,\xi)}=\frac{\hat F_{\xi^1 \xi^1}+\hat F_{\xi^2 \xi^2}}{\check F_{\xi^1 \xi^1}+\check F_{\xi^2 \xi^2}}\big|_{(x,\xi)}$ is well-defined and constant along the integral curves of the geodesic sprays of $\hat F$ and $\check F$.

		The Hessians of $F$ and $\tfrac12 F^2$ are related by $g_{ij}=F h_{ij}+F_{\xi^i}F_{\xi^j}$ and it follows by the positive definiteness of $g_{ij}$ and homogeneity, that $h_{ij}|_{(x,\xi)}\nu^i\nu^j=0$ if and only if $\nu$ is a multiple of $\xi$. Thus $\det h = 0$ and $\operatorname{tr} h \not= 0$, as otherwise $h$ would vanish.	
		
		Let $S= \xi^i \partial_{x^i} - 2G^i \partial_{\xi^i}$ be the geodesic spray of $\hat F$. By (\ref{EquationRapcsak}) for $F \in \{ \hat F, \check F\}$ we have
		$$F_{x^i}-F_{\xi^i x^\ell}\xi^\ell + 2G^\ell h_{i \ell}  =0,$$
		and thus by differentiating by $\xi^i$ and changing sign
		$$S(h_{ii}) -2G^\ell_i h_{i\ell}=0.$$
		Adding the two equations and using (\ref{EquationHessian}) gives
		$$S(\operatorname{tr} h)=2G_1^\ell h_{1\ell} + 2G^\ell_2 h_{2\ell}
		=\underbrace{2 \frac{G^1_1 (\xi^2)^2 - (G^1_2 + G^2_1)\xi^1 \xi^2 + G^2_2 (\xi^1)^2}{(\xi^1)^2 + (\xi^2)^2}}_{c(x,\xi):=} \operatorname{tr}h.$$
		As this is a linear ODE along the integral curves of $S$, any two solutions must be a constant multiple of each other along the integral curves. Let $\operatorname{tr} \hat h = I(x,\xi) \operatorname{tr}\check h$. Then
		$$c \operatorname{tr} \hat h = S(\operatorname{tr} \hat h) = S(I \operatorname{tr} \check h) = S(I) \operatorname{tr} \check h + cI \operatorname{tr} \check h = S(I)\operatorname{tr}\check h + c \operatorname{tr} \hat h,$$
		thus $S(I)=0$ as claimed.
		
		Now let us show that the function $I$ is well-defined.
		The value $I(x,\xi)$ is defined such that $\hat F_{\xi^i \xi^j}|_{(x,\xi)} = I(x,\xi) \check F_{\xi^i \xi^j}|_{(x,\xi)}$.
		Let $\overline x^i (x)$ be a change of coordinates. Then $\overline \xi^i(x,\xi) = \frac{\partial \overline x^i}{\partial x^j}\xi^j$ and $F_{\overline \xi^i \overline \xi^j} 
		= F_{\xi^k \xi^\ell} \frac{\partial \xi^k}{\partial \overline \xi^i}\frac{\partial \xi^\ell}{\partial \overline \xi^j}
		= F_{\xi^k \xi^\ell} \frac{\partial x^k}{\partial \overline x^i}\frac{\partial x^\ell}{\partial \overline x^j}$. 
		It follows that $\hat F_{\overline \xi^i \overline\xi^j}|_{(\overline x,\overline\xi)} = I(\overline x,\overline \xi) \check F_{\overline\xi^i \overline\xi^j}|_{(\overline x,\overline \xi)}$. Thus $I: T\mathcal S \backslash 0 \to \mathbb R$ is defined independent of the choice of coordinates.
		\end{proof}
	
	\section{Topological entropy and proof of \ref{FactPositiveEntropy}}\label{SectionPositiveEntropy}
	\begin{defn}Let $G$ be a group generated by finite set $S \subseteq G$. The group is of exponential growth, if for some $k>0$ it holds $\#B_n \geq C^{kn}$ for all $n\in \mathbb N$, where $\# B_n$ denotes the number of elements in $G$ that can be written as a product of at most $n$ elements from $S$ and inverses of those.
	\end{defn}
	It can be shown that this definition does not depend on the choice of the finite generator $S$.
	
	\begin{defn}Let $(X,d)$ be a compact metric space with symmetric distance function $d$ and $S^t: \mathbb R \times X \to X$ a flow. Define the family of distance functions $$d^t(x,y)=\max\limits_{0 \leq \tau \leq t} d(S^\tau(x),S^\tau(y)).$$
			 For $ \epsilon, t >0$ let $H^t_\epsilon$ be the maximal cardinality of $\epsilon$-separated sets in the metric space $(X,d^t)$.
			 The topological entropy of the flow $S^t$ is defined as
			\begin{equation*}\label{EquationDefinitionTopologicalEntropy}				
			h_{\operatorname{top}}= \lim_{\epsilon \to 0}  \limsup_{t \to \infty} \tfrac1t \log(H^t_\epsilon).
			\end{equation*}
	\end{defn}
	The topological entropy always exists and is a number in $[0, \infty]$. It does not depend on the choice of the metric $d$, but only on the induced topology.
	
%	\begin{lem}
%		The geodesic flow of any Finsler metric on a closed manifold $M$, whose fundamental group $\pi(M)$ is of exponential growth, has positive topological entropy.
%		\end{lem}

		Let us now proof proposition \ref{FactPositiveEntropy}: The geodesic flow of any Finsler metric on a compact manifold, whose fundamental group is of exponential growth, has positive topological entropy. 
		To deal with the irreversibility of the metrics, we shall use the reversibility number $\lambda_F := \sup\limits_{\xi \in TM} \frac{F(-\xi)}{F(\xi)} \geq 1$, which is finite if $M$ is compact as its unit sphere bundle is compact.
		
		Furthermore, let $d_F: M \times M \to \mathbb R$ the (possibly not symmetric) distance function of $F$, where $d_F(x,y)$ is defined as the infimum of the $F$-length of all curves from $x$ to $y$.
		Then the reversibility number of the $d_F$ given by $\lambda_{d_F}:=\sup\limits_{x,y \in M} \frac{d(y,x)}{d(x,y)}$ is at most $\lambda_F$.

		\begin{proof}[Proof of \ref{FactPositiveEntropy}]	
		Let $(\widetilde M, p: \tilde M \to M)$ be the universal cover of $M$. Let $d \mu$ be a volume form on $M$ invariant under $F$-isometries (e.g. the Holmes-Thompson or Busemann volume) and $d: M \times M \to \mathbb R$ the distance function induced by $F$. Let $\tilde F(\tilde \xi) := F(d \pi(\tilde \xi))$ be the lift of the Finsler metric to the universal cover $\widetilde M$, so that $p$ is a local isometry. Denote by $d \tilde \mu$ the corresponding volume form of $\tilde F$ and by $\tilde d$ the induced distance function on $\widetilde M$.
		
		\textbf{Firstly,} we show that the volume of closed forward balls $B_r(\tilde x)=\{\tilde y \in \widetilde M \mid \tilde d(\tilde x, \tilde y) \leq r\}$ in $\widetilde M$ grows exponentially with their radius as a consequence of the exponential growth of the fundamental group, that is for any $\tilde x \in \widetilde M$
		\begin{equation}\label{EquationExponentialGrowthOfBalls}\exists s_0,d_0,\mu_0, k >0 ~~ \forall  n \in \mathbb N:~~ \tilde \mu\big(B_{s_0 n+d_0}(\tilde x)\big) \geq \mu_0 e^{kn}.\end{equation}				
		By compactness and definition of the universal covering, there is a finite family of open connected, simply connected subsets $U_i \subseteq M$, that covers $M$ and such that $p^{-1}(U_i)$ is the union of open, disjoint subsets $V_{ij} \subseteq \widetilde M$, such that $p: (V_{ij}, \tilde F) \to (U_i,F)$ is an isometry.
		
		Fix $\tilde x \in \widetilde M$ and set $x := p(\tilde x)$. Let $S=\{a_1,..,a_\ell\}$ be a set of closed, smooth curves through $x$ generating the fundamental group $\pi(M,x)$ and assume that $S$ is closed under inversion. Let $\# B_n$ the number of elements that can we written as a product of at most $n$ elements from $S$. By assumption, there is $k>0$, such that $\#B_n \geq e^{k n}$ for all $n \in \mathbb N$. Set
		$$s_0 := \max_i \Big( \operatorname{length}_F(a_i) \Big) \qquad d_0 := \max_i \Big(\operatorname{diam}_d(U_i)\Big) \qquad
		\mu_0 := \min_i \Big( \mu(U_i) \Big).$$
		All three are positive. Let $\tilde x \in V_{ij}$ and note that $\tilde \mu(V_{ij})=\mu(U_i) \geq \mu_0$.
		For $a \in \pi(M,x)$, let $|a|$ be the smallest number of elements from $S$ whose product gives $a$.
		Consider the  covering transformation $\Gamma(a): V_{ij} \to \widetilde M$, that maps a $\tilde y \in V_{ij}$ to the endpoint of the unique lift of the curve $cac^{-1}$ starting at $\tilde y$, where $c$ is any curve from $y:=p(\tilde y)$ to $x$ inside $U_i$.
		Then
		$$\tilde d(\tilde x, \Gamma(a) \tilde y) \leq \inf_{\substack{c \text{ curve in }U_i\\\text{from $y$ to $x$ }}} \Big( \operatorname{length}(\tilde a) + \operatorname{length}(\tilde c^{-1})\Big) \leq s_0 |a| + d_0,$$
		where $\tilde a$ and $\tilde c^{-1}$ are the unique lifts of $a$ and $c^{-1}$ to $\widetilde M$ starting at $\tilde x$ and $\Gamma(a)\tilde x$ respectively, and have the same length as $a$ and $c^{-1}$ because $p$ is a local isometry. 
		Thus $\Gamma(a) V_{ij} \subseteq B_{s_0|a|+d_0}(\tilde x)$.
		Furthermore for different $a$ the sets $\Gamma(a)V_{ij}$ are disjoint and we have $p\big(\Gamma(a)V_{ij}\big)=U_i$, hence $\tilde \mu(\Gamma(a)V_{ij}) = \mu(U_i) \geq \mu_0$. It follows that
		$$\tilde \mu(B_{s_0 n + d_0}) \geq \tilde \mu\Big( \bigcup_{|a| \leq n} \Gamma(a)V_{ij}  \Big) \geq \mu_0 \cdot \# B_n \geq \mu_0 e^{kn}.$$
		
		%So far,  the constant $s_0$ and $c$ depend on the point $x$ and the choice of generator curves. 
		%By considering the symmetrization of $F$ and using compactness, there is a number $D>0$, such that each two points on $M$ can be joint by a curve, such that $\operatorname{length}(c)+\operatorname{length}(c^{-1}) \leq D$. Set $s'_0:=s_0 + D$.
		%Let $\tilde y \in \widetilde M$ be another point, $y:=p(\tilde y)$ and $c$ a curve from $x$ to $y$ as above. Then $\{ca_ic^{-1}\}$ is a generator of $\pi(M,y)$ and the length of each generating element is bounded by $s_0 + D =s_0'$. Hence using $s_0'$ we have established (\ref{EquationExponentialGrowthOfBalls}).
		
		\textbf{Secondly,} let $\rho$ and $\tilde \rho$ be the symmetrizations of $d$ and $\tilde d$ on $M$ and $\widetilde M$ respectively, that is $\rho( y_1,  y_2):=\frac{ d( y_1, y_2)+ d( y_2, y_1)}{2}$. and $\tilde \rho(\tilde y_1, \tilde y_2):=\frac{\tilde d(\tilde y_1,\tilde y_2)+\tilde d(\tilde y_2,\tilde y_1)}{2}$.
		Choose $\epsilon>0$ such that any $\rho$-ball of radius $2\epsilon$ in $M$ is contained in a set $U_i$, for example quarter of the Lebesgue number of the covering $U_i$ for the distance $\rho$. Then any $\tilde \rho$-ball of radius $2\epsilon$ in $\widetilde M$ is contained in one of the sets $V_{ij}$. In particular is the $\tilde \mu$-measure of $\tilde \rho$-balls of radius $2\epsilon$ bounded from above by a finite number $c_0>0$.
		
		 For fixed $\tilde x \in \widetilde M$, we show existence of a sequence $r_i \to \infty$, such that for each $r_i$ there are at least $\tfrac1{c_0} e^{\frac k 2 r_i}$ unit speed geodesics $\tilde \gamma_j^{r_i}$ of length $r_i$ starting from $\tilde x$, whose endpoints are $\epsilon$-seperated for the symmetrized distance $\tilde \rho$, where $k>0$ is as in the first part.
			
		Let $\delta >0$ and consider the $\tilde d$-annuli $U_{r}:=B_{r+\delta}(\tilde x) \backslash B_{r}(\tilde x)$. There is a sequence $r_i \to \infty$ such that $\tilde \mu(U_{r_i}) \geq e^{\tfrac k2 r_i}$. Indeed, suppose the inequality is violated for all but finitely many members of the sequence $r_i = i \delta$. Then $\tilde \mu (B_{n\delta}(\tilde x))= \sum_{i=0}^{n-1} \tilde \mu (U_{i\delta }) \leq \frac{2}{k \delta} e^{\tfrac k 2 \delta n} + C$, where $C$ is a constant independent of $n$. This contradicts (\ref{EquationExponentialGrowthOfBalls}).
		
		Let $Q_{r_i}$ be a maximal $2\epsilon$-separated set for $\tilde \rho$ in $U_{r_i}$. Then the $\tilde \rho$-balls of radius $2\epsilon$ with centers $\tilde q \in Q_{r_i}$ must cover $U_{r_i}$ and hence 
		$$c_0 \cdot \#Q_{r_i} \geq \tilde \mu(U_{r_i})\geq e^{\tfrac k2 r_i}.$$
		As $(M,F)$ is forward complete, so is $(\widetilde M, \tilde F)$ and for each $\tilde q \in Q_{r_i}$ we may choose a unit speed geodesic from $\tilde x$ to $\tilde q$ of length between $r_i$ and $r_i+\delta$. Let $\tilde \gamma_1, \tilde \gamma_2$ be two such geodesics ending at $\tilde q_1, \tilde q_2$. Then using $\tilde \rho(\tilde y_1, \tilde y_2) \leq \frac{1+\lambda}{2}\tilde d(\tilde y_1, \tilde y_2)$, where $\lambda$ is the reversibility number of $F$, we have
		\begin{align*}
		\tilde \rho(\gamma_1(r),\gamma_2(r)) & \geq \tilde \rho (\tilde q_1,\tilde q_2) - \tilde \rho (\tilde \gamma_1(r),q_1)- \tilde \rho (\tilde \gamma_1(r),q_2)	\\	
		& \geq 2\epsilon - \frac{1+\lambda}{2} \Big(\tilde d(\gamma_1(r),q_1) + \tilde d(\gamma_2(r),q_1)\Big)\\
		&\geq 2\epsilon - \frac{1+\lambda}{2} \cdot 2 \delta
\end{align*}		
		Thus choosing $\delta = \frac{\epsilon}{1+\lambda}$ gives the desired sequence $r_i$ and geodesics $\gamma^{r_i}_j$.
		
		\textbf{Finally,} let $\hat \rho$ be any symmetric distance on $TM \backslash 0$, such that $\hat \rho\big(\xi,\nu\big) \geq \rho\big(\pi(\xi), \pi(\nu)\big)$, where $\pi: TM \to M$ is the bundle projection.		
		By definition, the topological entropy of the geodesic flow is
		$h_{\operatorname{top}}= \lim\limits_{\epsilon \to 0}  \limsup\limits_{t \to \infty} \tfrac1t \log(H^t_\epsilon)$, where $H^t_{\epsilon}$ is the maximal cardinality of an $\epsilon$-separated set with respect to the distance $\hat \rho^t$ on $TM\backslash 0$ defined by $\hat \rho^t(\xi,\nu)=\max\limits_{0 \leq \tau \leq t} \hat \rho \big(S^\tau \xi, S^\tau \nu\big)$, where $S^\tau$ is the geodesic flow of $F$.

		 Let $r_i$ and $\gamma_j^{r_i}$ as before. Then the starting vectors of the to $M$ projected geodesics $\gamma^{r_i}_j := p(\tilde \gamma^{r_i}_j)$ are $\epsilon$-separated with respect to $\hat \rho^{r_i}$. Indeed, let $\gamma_1, \gamma_2$ be two such geodesics and $t\in (0,r_i]$ the smallest value, such that $\tilde \rho \big( \tilde \gamma_1(t), \tilde \gamma_2(t) \big)= \epsilon$. Then $\tilde \gamma_1(t)$ and $\tilde \gamma_2(t)$ lie in the same $V_{ij}$ and as $p$ is a local isometry also for the symmetrized distances, we have $ \rho \big(  \gamma_1(t), \gamma_2(t) \big)= \epsilon$.
		Because $H^t_{\epsilon}$ is monotonously increasing as $\epsilon \to 0$, it follows that
		$$h_{\operatorname{top}} \geq \limsup\limits_{r_i \to \infty} \tfrac1{r_i} \log(\tfrac{1}{c_0} e^{\tfrac k2 r_i}) \geq \tfrac k2>0.$$
		\end{proof}

	%\textbf{Acknowledgement.} 

	\bibliographystyle{plain}
	\bibliography{literature}

\begin{thebibliography}{10}

\bibitem{alvarez}
J.~\'{A}lvarez Paiva and G.~Berck.
\newblock Finsler surfaces with prescribed geodesics.
\newblock {\em ArXiv e-prints}, February 2010.

\bibitem{Arcostanzo}
M.~Arcostanzo.
\newblock Des m\'etriques finsl\'eriennes sur le disque \`a partir d'une
  fonction distance entre les points du bord.
\newblock In {\em S\'eminaire de {T}h\'eorie {S}pectrale et {G}\'eom\'etrie,
  {N}o.\ 10, {A}nn\'ee 1991--1992}, volume~10 of {\em S\'emin. Th\'eor. Spectr.
  G\'eom.}, pages 25--33. Univ. Grenoble I, Saint-Martin-d'H\`eres, 1992.

\bibitem{CrampinMestdagSaundersTheMultiplierApproach}
M.~Crampin, T.~Mestdag, and D.~J. Saunders.
\newblock The multiplier approach to the projective {F}insler metrizability
  problem.
\newblock {\em Differential Geom. Appl.}, 30(6):604--621, 2012.

\bibitem{DinaburgOnTheRelationsAmongVariousEntropy}
E.~Dinaburg.
\newblock On the relations among various entropy characteristics of dynamical
  systems.
\newblock {\em Mathematics of the {USSR}-Izvestiya}, 5(2):337--378, apr 1971.

\bibitem{ManningTopologicalEntropyForGeodesicFlows}
A.~Manning.
\newblock Topological entropy for geodesic flows.
\newblock {\em Ann. of Math. (2)}, 110(3):567--573, 1979.

\bibitem{MatveevTopalovTrajectoryEquivalence}
V.~Matveev and P.~Topalov.
\newblock Trajectory equivalence and corresponding integrals.
\newblock {\em Regul. Chaotic Dyn.}, 3(2):30--45, 1998.

\bibitem{TopalovMatveevMetricWithErgodic}
V.~Matveev and P.~Topalov.
\newblock Metric with ergodic geodesic flow is completely determined by
  unparameterized geodesics.
\newblock {\em Electron. Res. Announc. Amer. Math. Soc.}, 6:98--104, 2000.

\bibitem{TopalovMatveevGeodesicEquivalenceViaIntegrability}
V.~Matveev and P.~Topalov.
\newblock Geodesic equivalence via integrability.
\newblock {\em Geom. Dedicata}, 96:91--115, 2003.

\bibitem{MettlerGeodesicRigidity}
T.~Mettler.
\newblock Geodesic rigidity of conformal connections on surfaces.
\newblock {\em Math. Z.}, 281(1-2):379--393, 2015.

\bibitem{MettlerPaternainConvexProjective}
T.~Mettler and G.~Paternain.
\newblock Convex projective surfaces with compatible weyl connection are
  hyperbolic, 2018.
\newblock arXiv:1804.04616.

\bibitem{PaternainEntropyAndCompletelyIntegrableHamiltonianSystems}
G.~Paternain.
\newblock Entropy and completely integrable {H}amiltonian systems.
\newblock {\em Proc. Amer. Math. Soc.}, 113(3):871--873, 1991.

\bibitem{PaternainFinslerStructures}
G.~Paternain.
\newblock Finsler structures on surfaces with negative {E}uler characteristic.
\newblock {\em Houston J. Math.}, 23(3):421--426, 1997.

\end{thebibliography}

\end{document}